\documentclass[a4paper,11pt]{article}

\usepackage{amsmath,amssymb,amsthm}
\usepackage{graphicx,url}
\usepackage[curve,knot]{xypic}

\usepackage[a4paper,left=1in,top=1in,right=1in,bottom=1in]{geometry}

\usepackage{stmaryrd,bbm}

\theoremstyle{plain}
\newtheorem{theorem}{Theorem}[section]
\newtheorem{proposition}[theorem]{Proposition}
\newtheorem{lemma}[theorem]{Lemma}
\newtheorem{corollary}[theorem]{Corollary}
\theoremstyle{definition}
\newtheorem{definition}{Definition}[section]
\theoremstyle{remark}

\newtheorem*{example*}{Example}
\newenvironment{jegnote}[1]{\vspace{1em}\par \noindent \textit{Note{#1}:}}{\vspace{1em}}

\input cyracc.def

\renewcommand{\labelenumi}{\textit{\theenumi}\textup{)}}

\newcommand{\Ad}{\mathrm{Ad}}
\newcommand{\Adj}[2]{{\Ad_{#1}(#2)}}
\newcommand{\bantipode}{\underline{\;\,}\mkern-8mu{S}}
\newcommand{\bigdsum}{\bigoplus}
\newcommand{\bcoproduct}{\underline{\;\;}\mkern-12mu{\coproduct}}
\newcommand{\bcounit}{\underline{\;\,}\mkern-8mu{\varepsilon}}

\newcommand{\brprim}[1]{\underline{\mathrm{Prim}}( #1 )}

\newcommand{\breta}{\underline{\;\,}\mkern-8mu{\eta}}
\newcommand{\btensor}{\mathrel{\underline{\;\;}\mkern-12mu{\tensor}}}
\newcommand{\Cdot}{\mathrel{\cdot}}
\newcommand{\Cdotp}{\mathrel{\cdot'}}
\let\chisave\chi
\renewcommand{\chi}{{%
 \mathchoice{\raisebox{0.25ex}{$\displaystyle\chisave$}}
            {\raisebox{0.2ex}{$\textstyle\chisave$}}
            {\raisebox{0.2ex}{$\scriptstyle\chisave$}}
            {\raisebox{0.1ex}{$\scriptscriptstyle\chisave$}}}}

\newcommand{\complex}{\ensuremath \mathbb{C}}
\newcommand{\coproduct}{\ensuremath \Delta}

\newcommand{\counit}{\ensuremath \varepsilon}
\newcommand{\cross}{\times}
\newcommand{\curly}[1]{\ensuremath{\mathcal{#1}}}
\newcommand{\DBos}[3]{#1 \rtimesdot #2 \ltimesdot #3}

\newcommand{\dbxsum}{\bowtie}

\newcommand{\defeq}{\stackrel{\scriptscriptstyle{\mathrm{def}}}{=}}
\newcommand{\dsum}{\ensuremath{ \oplus}}
\newcommand{\dual}[1]{\ensuremath {#1}^{*}}
\renewcommand{\epsilon}{\varepsilon}
\newcommand{\Hom}[3]{\ensuremath \mbox{Hom}_{#1}(#2,#3)}
\newcommand{\id}{\ensuremath \mbox{\textup{id}}}

\newcommand{\Image}[1]{\ensuremath \mbox{Im}\>{#1}}

\newcommand{\inj}{\hookrightarrow}
\newcommand{\integ}{\ensuremath{\mathbb{Z}}}
\newcommand{\intersection}{\mathrel{\cap}}
\newcommand{\inv}[1]{\ensuremath {#1}^{-1}}
\newcommand{\ip}[2]{\ensuremath \lgen\;\!#1,#2\;\!\rgen}

\newcommand{\iso}{\ensuremath \cong}

\newcommand{\laction}{\triangleright}
\newcommand{\lgen}{\ensuremath \mathopen{<}} 

\newcommand{\ltimesdot}{\mathrel{\cdot\joinrel\mkern-14.3mu\rhd\mkern-8.7mu<}}
\newcommand{\Lbracket}[2]{\ensuremath{ [\, #1 , #2\, ]}}
\newcommand{\Lie}[1]{\ensuremath{\mathfrak{#1}}}

\newcommand{\Mor}[3]{\text{Mor}_{#1}(#2 , #3 )}
\newcommand{\mult}[2]{\mathrm{mult}_{#1}( #2 )}
\newcommand{\nat}{\ensuremath \mathbb{N}}

\newcommand{\onto}{\twoheadrightarrow}
\newcommand{\op}[1]{{#1}^{\mbox{\scriptsize \textup{op}}}}
\renewcommand{\phi}{\varphi}
\newcommand{\qbin}[3]{{\genfrac{[}{]}{0pt}{0}{#1}{#2}}_{#3}}
\newcommand{\qbracket}[3]{\ensuremath{ [\, #1 , #2\, ]_{#3}}}
\newcommand{\qea}[1]{\ensuremath{U_{q}(\Lie{#1})}}
\newcommand{\qeaminus}[1]{\ensuremath{U_{q}^{-}(\Lie{#1})}}
\newcommand{\qeaplus}[1]{\ensuremath{U_{q}^{+}(\Lie{#1})}}
\newcommand{\qeb}[1]{\ensuremath{U_{q}^{{\scriptscriptstyle \leqslant}}(\Lie{#1})}}
\newcommand{\qebO}[1]{\ensuremath{U_{q}^{{\scriptscriptstyle \leqslant}}(\Lie{#1})_{[0]}}}
\newcommand{\qebplus}[1]{\ensuremath{U_{q}^{{\scriptscriptstyle \geqslant}}(\Lie{#1})}}

\newcommand{\qint}[2]{{\left[ #1 \right]}_{#2}}
\newcommand{\raction}{\triangleleft}
\newcommand{\rgen}{\ensuremath \mathclose{>}}
\newcommand{\rip}[2]{\mathopen{(} #1,#2 \mathclose{)}}
\renewcommand{\rtimes}{\mathrel{>\joinrel\mkern-6mu\lhd}}
\newcommand{\rtimesdot}{\mathrel{>\joinrel\mkern-6mu\lhd\mkern-7.2mu\cdot}}

\newcommand{\tensor}{\ensuremath \otimes}
\newcommand{\too}{\longrightarrow}
\newcommand{\union}{\mathrel{\cup}}
\newcommand{\weight}[2]{\mathop{\mathrm{wt}}_{ #1 }( #2 )}

\newcommand{\cf}{cf.\ }
\newcommand{\ie}{i.e.\ }
\newcommand{\eg}{e.g.\ }

\title{Braided enveloping algebras associated to \\ quantum parabolic subalgebras}
\author{Jan E. Grabowski\footnotemark[2] 
\\ \small{\textit{Mathematical Institute, University of Oxford}} \\ \small{\textit{24-29 St.\ Giles', Oxford, OX1 3LB, United Kingdom}}}
\date{27th July 2010}

\begin{document}

\maketitle

\renewcommand{\thefootnote}{\fnsymbol{footnote}}
\footnotetext[2]{Email: \url{jan.grabowski@maths.ox.ac.uk}.  Website: \url{http://people.maths.ox.ac.uk/~grabowsk/}}
\renewcommand{\thefootnote}{\arabic{footnote}}
\setcounter{footnote}{0}

\begin{abstract} 
\noindent Associated to each subset $J$ of the nodes $I$ of a Dynkin diagram is a triangular decomposition of the corresponding Lie algebra $\Lie{g}$ into three subalgebras $\widetilde{\Lie{g}_{J}}$ (generated by $e_{j}$, $f_{j}$ for $j\in J$ and $h_{i}$ for $i\in I$), $\Lie{n}^{-}_{D}$ (generated by $f_{d}$, $d\in D=I\setminus J$) and its dual $\Lie{n}_{D}^{+}$.

We demonstrate a quantum counterpart, generalising work of Majid and Rosso, by exhibiting analogous triangular decompositions of $\qea{g}$ and identifying a graded braided Hopf algebra that quantizes $\Lie{n}_{D}^{-}$.  This algebra has many similar properties to $\qeaminus{g}$, in many cases being a Nichols algebra and therefore completely determined by its associated braiding.
\end{abstract}

\noindent \footnotesize \textbf{Keywords:} quantized enveloping algebra, braided Hopf algebra, Nichols algebra \normalsize \hspace*{\fill} \linebreak

\section{Introduction}\label{s:intro}

It is now twenty-five years since the study of quantum groups began in earnest and much of the attention in the area has been focused on the quantized enveloping algebras introduced by Drinfel\cprime d (\cite{DrinfeldHAQYBE}) and Jimbo (\cite{Jimbo}) and their structure as illuminated particularly by Lusztig (\cite{LusztigBook}).  However, a significant amount of development has taken place in other settings inspired by quantum theory, especially the study of non-commutative versions of classical algebraic and geometric objects obtained by introducing braidings and braided categories.  It has long been known that aspects of the theory of quantized enveloping algebras have natural statements in the language of braided categories.  Conversely, when studying braided structures, one finds Lie-theoretic type information---particularly Cartan matrices---appearing very naturally.  One recent example would be the work of Andruskiewitsch and Schneider (\cite{AndrSchnPHA}) and others on pointed Hopf algebras.  In the present work, the relationship between quantized enveloping algebras and Hopf algebra structures in braided categories is examined further.

In a series of papers, Majid (\cite{MajidInductive},\cite{MajidNewConstr},\cite{DoubleBos}) has introduced a construction for Hopf algebras called double-bosonisation.  A special case has also been defined by Sommerh\"{a}user (\cite{Sommerhauser}).  The input is a Hopf algebra and two braided Hopf algebras in duality; the output is a new Hopf algebra.  This construction is on the one hand related to particular biproducts of a Hopf algebra and a braided Hopf algebra, called bosonisations.  Bosonisations are semi-direct products of Hopf algebras and double-bosonisations are a form of triple product built from two bosonisations.  On the other hand, double-bosonisation is modelled on and generalises triangular decompositions of the type seen in Lie theory and its quantum counterpart.  Indeed, in his original work Majid showed that the triangular decomposition of a quantized enveloping algebra into positive, negative and Cartan parts is one way of expressing the quantized enveloping algebra as a double-bosonisation.  (This may be found in \cite{DoubleBos}; an extended exposition is in \cite[Chapters 17-19]{QGP}.)

Along with establishing the double-bosonisation as a Hopf algebra and also exhibiting the quantized enveloping algebras as an example in the manner just discussed, Magid (\cite{MajidInductive}) introduced the idea that double-bosonisation allows for an alternative approach to the study of quantized enveloping algebras.  One can think of double-bosonisation as realising in the algebraic structure the addition of nodes to Dynkin diagrams and so as allowing the inductive construction of quantized enveloping algebras.  In particular, he saw that the inductive construction along the $A$ series of Dynkin diagrams can be achieved using braided (hyper-)planes.  These are among the simplest of the braided Hopf algebras and are in some sense non-commutative vector spaces.  He went on to observe that one is not restricted to the $A$ series and a more general consideration was possible.  

At around the same time, Rosso (\cite{Rosso}) considered a similar construction, using quantum symmetric algebras over irreducible modules to add a single node to a Dynkin diagram, now not just of type $A$.  In this work, we take this up and provide a more formal analysis, showing that this idea of induction indeed applies very generally: not just in the corank one case, nor just along the Dynkin series or even just in finite type but for arbitrary root data, where the associated modules need not be irreducible or finite-dimensional, and that Nichols algebras (now the more commonly used term for quantum symmetric algebras) are precisely what is needed to make this sort of inductive construction work.

The principal aim of this work was to extend that of Majid (\cite{BraidedLie},\cite{DoubleBos}) and our own (\cite{LieInduction}) to the quantum setting, so giving the most general setting for the above ideas of Majid and Rosso.  In addition, relatively few examples of infinite-dimensional braided Hopf algebras are well understood, which was another motivation for our work.  

We briefly summarize the classical version of the idea presented here.  Associated to every subset $J$ of the set of nodes $I$ of a Dynkin diagram is a standard parabolic subalgebra $\Lie{p}_{J}$ of the corresponding Lie algebra $\Lie{g}$, generated by the positive Borel subalgebra of $\Lie{g}$ together with the negative simple generators $f_{j}$, $j\in J$.  Furthermore one has a decomposition of $\Lie{g}$ as a semi-direct product of $\Lie{p}_{J}$ and the subalgebra $\Lie{n}^{-}_{D}$ generated by the remaining negative simple generators $f_{d}$, $d\in D=I\setminus J$.  One also has $\Lie{g}_{J}$, the Lie algebra generated by the $e_{j}$ and $f_{j}$ with $j\in J$.  In \cite{LieInduction} and \cite{BraidedKacMoody}, we showed that $\Lie{g}$ has a semi-direct product decomposition into three subalgebras $\Lie{n}_{D}^{-}$, a central extension of $\Lie{g}_{J}$ and $\Lie{n}_{D}^{+}$ (dual to $\Lie{n}_{D}^{-}$).  This generalises the usual triangular decomposition into negative, Cartan and positive parts, which is of course a special case.

Lusztig's approach to quantized enveloping algebras starts with a root datum $\mathfrak{T}$, an abstraction of the notion of a root system for Lie algebras.  The corresponding quantized enveloping algebra will be denoted here by $\qea{T}$.  In order to have an abstract description of the choice of a standard parabolic subalgebra we define a relation between pairs of root data which we call a sub-root datum (Definition~\ref{s:defnofB:subrootdatum}), denoted $\mathfrak{J} \subseteq_{\iota} \mathfrak{T}$.  The conditions of the definition require that the Dynkin diagram of $\mathfrak{J}$ is a sub-diagram of that of $\mathfrak{T}$ but also impose further compatibility constraints.

The analogue of a negative Borel subalgebra in the quantum setting, denoted $\qeb{T}$, has a natural $\nat$-grading coming from the sub-root datum $\mathfrak{J} \subseteq_{\iota} \mathfrak{T}$ and the Radford--Majid theorem applies to give us a braided Hopf algebra $B=B(\mathfrak{T},\mathfrak{J},\iota,q)$.  We see that the zeroth component of this grading is a semi-direct product of $\qeb{J}$, the quantum negative Borel subalgebra associated to the `smaller' root datum, by a group Hopf algebra.  Then we see that $\qea{T}$ is indeed a double-bosonisation of $\qea{J}$ (Theorem~\ref{s:defnofB:isotoDB}).  Note that the usual triangular decomposition is obtained by considering the inclusion of the trivial (rank zero) datum in a given root datum.

In Section~\ref{s:Bstr} we analyse the algebra, module and coalgebra structures of $B$.  We give a set of generators for $B$ (Theorem~\ref{s:Bstr:alggensofB}), show that its first homogeneous component $B_{1}$ is a direct sum of quotients of Weyl modules and show that $B$ is integrable.  In the generic situation, this should be viewed as the dual of the quantum version of the Pl\"{u}cker embedding for partial flag varieties and their big cells.

Using the description of the generators of $B$, we prove that if $B_{1}$ is finite-dimensional then $B$ is a Nichols algebra (Theorem~\ref{s:Bstr:BisNichols}).  That is, $B$ is a graded braided Hopf algebra generated in degree one with all its braided-primitive elements also in degree one.  We refer the reader to the survey of Takeuchi (\cite{Takeuchi}) for more information on general Nichols algebras.

Restricting to the finite-type case, one question remains: does the graded dual of $B$ have a natural non-commutative geometric interpretation?  An affirmative answer to this has been given by Kolb (\cite{Kolb}) who has shown that the graded dual of the quotient $\qea{g}/\qea{p_{\mathit{J}}}$ is isomorphic to the quantized coordinate ring $\curly{O}_{q}[N^{-}_{D}]$ (where $N^{-}_{D}$ is the opposite unipotent radical to the parabolic subgroup $P_{J}$ of the group $G$, all constructed in analogy to their Lie counterparts described above).  That is, the object of our study here, $B$, should be regarded as the quantized enveloping algebra of $\Lie{n}_{D}^{-}$.

\section{Preliminaries}\label{s:prelims}

Throughout, we will use the following convention for the natural numbers: $\nat= \{ 0, 1, 2, \ldots \}$, setting $\nat^{*}= \{ 1, 2, 3, \ldots \}$.  That is, for us $\nat$ is a monoid.

Recall that a braided category is a monoidal category together with a natural isomorphism $\Psi\colon -\tensor - \to -\mathrel{\op{\tensor}} -$ (where $A \mathrel{\op{\tensor}} B=B\tensor A$), satisfying suitable identities (see for example \cite{MacLane} or \cite{QGP}).  One can consider objects in categories with various sorts of algebraic structures on them.  That is, one takes an object together with some morphisms from the category that satisfy the axioms for the appropriate algebraic structure, when we translate axioms into identities of morphisms.  A key example is that of a Hopf algebra in a braided category, namely an object $B$ and morphisms $\underline{m},\breta,\bcoproduct,\bcounit$ and $\bantipode$ satisfying the usual relations for an algebra product, unit, Hopf algebra coproduct, counit and antipode respectively.  In particular $\bcoproduct\colon B \to B \btensor B$ is required to be a morphism of braided algebras from $B$ to the braided tensor product algebra $B\btensor B$ (where the usual tensor product multiplication is twisted by $\Psi$, the braiding in the category).  We will also use the term `braided Hopf algebra' for a Hopf algebra in a braided category.

If $B=(B,\underline{m},\breta,\bcoproduct,\bcounit,\bantipode)$ is a braided Hopf algebra, we say $b \in B$ is braided-primitive if $\bcoproduct b = b \btensor 1 + 1 \btensor b$.  We will denote the vector space of braided-primitive elements of $B$ by $\brprim{B}$.  We note that $\brprim{B}$ is not in general a subalgebra of $B$.

Now we consider graded Hopf algebras.  Let $(M,+)$ be a commutative monoid, with identity element denoted $0$, and let $k[M]$ be the associated monoid algebra over a field $k$.  An $M$-graded $k$-Hopf algebra $H=\bigdsum_{m\in M} H_{m}$ is a $k$-Hopf algebra in the category of right $k[M]$-comodules, $\mathcal{M}^{k[M]}$.  (It is straightforward to recover the usual formulation of a grading for \eg $M=\integ$.)

Nichols algebras, also called Nichols--Woronowicz algebras were introduced in Nichols' thesis (\cite{Nichols}); Woronowicz (\cite{Woron}) and others have independently re-discovered them.

\begin{definition}[{\cite{AndrSchnPHA}}]\label{s:prelims:d:Nicholsalg}
A Nichols algebra is an $\nat$-graded braided $k$-Hopf algebra $B=\dsum_{n\in \nat} B_{n}$ such that:
\renewcommand{\labelenumi}{\textup{(\hspace{-1pt}}\textit{\alph{enumi}}\textup{\hspace{-0.5pt})}\ }
\begin{enumerate}
\setlength{\itemsep}{0.5pt}
\item $B_{0}=k$,
\item $B_{1}=\brprim{B}$, and
\item $B$ is generated as an algebra by $B_{1}$.
\end{enumerate}
\renewcommand{\labelenumi}{\textit{\roman{enumi}}\textup{)}\ }
\end{definition}
  
Some examples of Nichols algebras arise as braided versions of the classical symmetric or exterior algebras.  Others, notably those in this work, are analogues of enveloping algebras.  A good introduction to Nichols algebras may be found in \cite{AndrNichols}.

\begin{lemma}\label{s:prelims:l:Nicholsifdualdeg1gen}
Let $R=\dsum_{n \in \nat} R_{n}$ be a graded $k$-Hopf algebra in a braided category $\curly{C}$ with finite-dimensional homogeneous components.  Assume that $S=\dsum_{n\in \nat} \dual{R_{n}}$, the graded dual of $R$, is also a Hopf algebra in the braided category $\curly{C}$ (with the dual Hopf algebra structures).  Further, assume that $R_{0}=k$, so then $S_{0}=k$ also.  Then $R_{1}=\brprim{R}$ if and only if $S$ is generated as an algebra by $S_{1}$. \qed
\end{lemma}

\noindent This lemma is proved exactly as Lemma~5.5 of \cite{AndrSchnFinQGps}, from which it is derived; the proof there uses only properties of the graded braided Hopf algebra structures.

\subsection{The bosonisation constructions for Hopf algebras}\label{s:prelims:ss:HAbos}

Bosonisation and double-bosonisation are the two key constructions which make an inductive approach to the study of the quantized enveloping algebras possible.  Bosonisation takes a Hopf algebra and a braided Hopf algebra in its category of modules and combines these to obtain a new (ordinary) Hopf algebra.  Double-bosonisation incorporates the dual of the braided Hopf algebra as well and again produces a Hopf algebra.

However, bosonisation and double-bosonisation require an additional condition on the initial Hopf algebra $H$ which forms the input into the constructions.  This condition is the existence of a weak quasitriangular structure on $H$ and $H'$ dually paired to $H$.  The existence of a weak quasitriangular structure is, as the name suggests, a weaker condition than quasitriangularity.

\begin{definition}[\cf {\cite{MajidBiXprod}}]\label{s:prelims:weakqt}
Let $H$ and $H'$ be dually paired $k$-Hopf algebras, paired by the map $\ip{\ }{\ }\colon H \tensor H' \to k$.  A weak quasitriangular system consists of $H$, $H'$ and a pair of convolution-invertible algebra and anti-coalgebra maps $\curly{R}$, $\bar{\curly{R}} \colon H' \to H$, with convolution-inverses $\inv{\curly{R}}$, $\inv{\bar{\curly{R}}}$ respectively, such that
\begin{enumerate}
\item $\ip{\bar{\curly{R}}(\phi)}{\psi}=\ip{\inv{\curly{R}}(\psi)}{\phi}$ for all $\psi,\, \phi \in H'$ and
\item $\curly{R}$ and $\bar{\curly{R}}$ intertwine the left and right coregular actions $\dual{L}$, $\dual{R}$ with respect to the convolution product ${}\Cdot{}$ on $\Hom{k}{H'}{H}$:
\begin{align*} & \dual{L}(h)(a) \defeq \ip{h_{(1)}}{a}h_{(2)},\ \dual{R}(h)(a) \defeq h_{(1)}\ip{h_{(2)}}{a} \\
& \dual{R}(h) = \curly{R} \Cdot \dual{L}(h) \Cdot \inv{\curly{R}} \\
& \dual{R}(h) = \bar{\curly{R}} \Cdot \dual{L}(h) \Cdot \inv{\bar{\curly{R}}}
\end{align*} where we consider $\dual{L}\colon H' \tensor H \to H$ as a map $\dual{L}(h)\colon H' \to H$ by fixing $h \in H$ (similarly for $\dual{R}$).
\end{enumerate}
\end{definition}

\noindent We will denote by $WQ(H,H',\curly{R},\bar{\curly{R}})$ a weak quasitriangular system with the above data.

We can now define the bosonisation construction.  This was originally introduced in \cite{MajidCrossProducts}, where the claims implicit in the definition are proved, and it is noted in \cite{DoubleBos} that one needs only weak quasitriangularity for the construction to work.  We have altered the presentation slightly to reflect our use of the definition of a weak quasitriangular system.

\begin{definition}\label{s:prelims:bosonisation}
Let $WQ(H,H',\curly{R},\bar{\curly{R}})$ be a weak quasitriangular system and let $B$ be a Hopf algebra in the braided category of right $H'$-comodules, $\curly{M}^{H'}$.  Let the right coaction be denoted $\beta \colon B \to B \tensor H'$, $\beta(b)=b^{(1)} \tensor b^{(2)}$.  Then the bosonisation of $B$, denoted $B \rtimesdot H$, is the Hopf algebra with
\begin{enumerate}
\setlength{\itemsep}{0.5pt}
\item underlying vector space $B \tensor H$,
\item semi-direct product by the action\/ $\laction$ given by evaluation against the right coaction of $H'$:
\begin{align*} & (b \tensor h)(c \tensor g) = b(h_{(1)} \laction c) \tensor h_{(2)}g \\
& h \laction b = b^{(1)}\ip{h}{b^{(2)}}\quad \forall\ h\in H,\ b\in B 
\end{align*}
\item semi-direct coproduct by the coaction $\alpha$ of $H$ induced by the right coaction of $H'$ and the weak quasitriangular structure:
\[ \coproduct(b \tensor h) = b_{\underline{(1)}} \tensor \curly{R}({b_{\underline{(2)}}}^{(2)})h_{(1)} \tensor {b_{\underline{(2)}}}^{(1)} \tensor h_{(2)} \]
\item tensor product unit and counit and
\item an antipode, given by an explicit formula which we omit.
\end{enumerate}
\end{definition}

\noindent This is the left-handed version of bosonisation; the right-handed version is entirely analogous and is denoted $H \ltimesdot B$.  Double-bosonisation is defined by combining a left and a right bosonisation, with some cross relations.  A large part of \cite{DoubleBos} is devoted to showing that this is well-defined and produces a Hopf algebra.

\begin{definition}[{\cf \cite{DoubleBos}}]\label{s:prelims:dbosdefn}
Let $WQ(H,H',\curly{R},\bar{\curly{R}})$ be a weak quasitriangular system and let $B$ be a Hopf algebra in the braided category of right $H'$-comodules, $\curly{M}^{H'}$.  Let $B'$ be another Hopf algebra in the braided category $\curly{M}^{H'}$ with an invertible braided antipode dually paired with $B$ via $\mathrm{ev}\colon B \tensor B' \to k$, a dual pairing of the Hopf algebra structures in this category.  Then the double-bosonisation $\DBos{B}{H}{\op{(B')}}$ of $B$ and $\op{(B')}$ by $H$ is the Hopf algebra with
\begin{enumerate}
\item underlying vector space $B \tensor H \tensor B'$,
\item sub-Hopf algebras $B \rtimesdot H\ (\equiv B \rtimesdot H \tensor 1)$, $H \ltimesdot \op{(B')}$, and
\item cross-relations 
\[ b_{\underline{(1)}}\curly{R}({b_{\underline{(2)}}}^{(2)})c_{\underline{(1)}}\mathrm{ev}({c_{\underline{(2)}} \tensor b_{\underline{(2)}}}^{(1)} ) 
= \mathrm{ev}( {c_{\underline{(1)}}} \tensor {b_{\underline{(1)}}}^{(1)} )c_{\underline{(2)}}{\bar{\curly{R}}}({b_{\underline{(1)}}}^{(2)}){b_{\underline{(2)}}} \]
for all $b \in B$, $c\in \op{(B')}$.
\end{enumerate}
\end{definition}

\subsection{Root data and quantized enveloping algebras}\label{s:prelims:ss:rootdata+QEAs}

We follow Lusztig (\cite{LusztigBook}) in working with Cartan data and root data.  Consider a root datum $\mathfrak{T}=(I,{}\Cdot{},Y,X,\ip{\ }{\ },i_{1}\colon I \inj Y,i_{2}\colon I\inj X)$.  That is, $I$ is a finite set, ``${}\Cdot{}$'' is a symmetric bilinear form on $\integ[I]$ giving rise to an associated Cartan matrix $C$, $Y$ and $X$ are two finitely generated free Abelian groups perfectly paired by $\ip{\ }{\ }$ and $i_{1}$, $i_{2}$ are inclusions of $I$ into $Y$ and $X$ such that $\ip{i_{1}(i)}{i_{2}(j)}=C_{ij}$.  The root lattice is embedded into $X$, the weight lattice in $Y$.  We set $c_{i} \defeq \frac{i\Cdot i}{2}$.

Throughout, we will restrict our consideration to $q \in \dual{k}$ such that $q$ is not a root of unity, although we allow our base field $k$ to have arbitrary characteristic.  (In doing so, we will be making use of the non-degeneracy of certain forms and the validity of this in this generality is described in \cite[Chapter~8]{Jantzen}.)  Let $q_{i} \defeq q^{c_{i}}$ and let $\qint{a}{i}$ denote the $a$th symmetric $q_{i}$-integer and $\qbin{n}{k}{i}$ the corresponding $q_{i}$-binomial coefficient (see \eg \cite[Chapter~0]{Jantzen}).

Let the identity element of $Y$ be denoted $0$ and let $Z$ denote the free Abelian subgroup $\integ [i_{1}(I)]$ of $Y$.  We consider $I$ as a subset of $Z$, suppressing the map $i_{1}$.  We can now define the quantized enveloping algebra $\qea{T}$ associated to the root datum $\mathfrak{T}$ over the field $k$ with deformation parameter $q$.

\begin{definition}\label{s:prelims:qea}
We define $\qea{T}$ to be the Hopf algebra over $k$ generated by $E_{i}$, $F_{i}$ \textup{(}$i\in I$\textup{)} and $K_{\mu}$ \textup{(}$\mu \in Z$\textup{)}, subject to relations
\begin{description}
\item[(R1)] $K_{0}=1$, $K_{\mu}K_{\nu}=K_{\mu+\nu}$
\item[(R2)] $K_{\mu}E_{i}=q^{\ip{\mu}{i_{2}(i)}}E_{i}K_{\mu}$
\item[(R3)] $K_{\mu}F_{i}=q^{-\ip{\mu}{i_{2}(i)}}F_{i}K_{\mu}$
\item[(R4)]
  $\displaystyle{E_{i}F_{j}-F_{j}E_{i}=\delta_{ij}\frac{H_{i}-\inv{H_{i}}}{q_{i}-\inv{q_{i}}}}$, where $H_{i} \defeq K_{i}^{c_i}$
\item[(R5)] $\displaystyle\sum_{m=0}^{1-C_{ij}} (-1)^{m}\qbin{1-C_{ij}}{m}{i} E_{i}^{1-C_{ij}-m}E_{j}E_{i}^{m}=0$, for $i\neq j$
\item[(R6)] $\displaystyle\sum_{m=0}^{1-C_{ij}} (-1)^{m}\qbin{1-C_{ij}}{m}{i} F_{i}^{1-C_{ij}-m}F_{j}F_{i}^{m}=0$, for $i\neq j$
\end{description}

\noindent The Hopf structure is: \rule{0pt}{1.45em} 
\begin{align*}
  & \coproduct E_{i}   = E_{i}\tensor 1 + H_{i}\tensor E_{i}       & & \counit(E_{i})   = 0 & & SE_{i}   = -\inv{H_{i}}E_{i} \\
  & \coproduct F_{i}   = F_{i}\tensor \inv{H_{i}} + 1\tensor F_{i} & & \counit(F_{i})   = 0 & & SF_{i}   = -F_{i}H_{i} \\
  & \coproduct K_{\mu} = K_{\mu} \tensor K_{\mu}                           & & \counit(K_{\mu}) = 1 & & SK_{\mu} = \inv{K_{\mu}}
\end{align*}
\end{definition}

\begin{jegnote}{} 
\begin{enumerate}
\item There are several definitions of the quantized enveloping algebras in the literature and this one is close to that of Lusztig (\cite{LusztigBook}), except that he has generators $K_{\mu}$ with $\mu \in Y$, rather than referring to $Z$.  Our definition also resembles that of Jantzen (\cite{Jantzen}), although he starts with root systems, rather than root data.  The reason for the restriction to generators indexed by the subgroup $Z$ rather than $Y$ is technical and is discussed below.
\item Since $Z$ is finitely generated and we could define $\qea{T}$ using only $K_{i}$, $i \in I$, so this version of $\qea{T}$ is finitely generated.
\item Also useful will be the following relations, implied by \textbf{(R2)} and \textbf{(R3)}:
 \begin{description}
 \item[(R2\cprime )] $H_{i}E_{j} = q^{\ip{c_{i}i_{1}(i)}{i_{2}(j)}}E_{j}H_{i} =q^{c_{i}C_{ij}}E_{j}H_{i}  =q^{i\Cdot j}E_{j}H_{i}$
 \item[(R3\cprime )] $H_{i}F_{j}=q^{-i\Cdot j}F_{j}H_{i}$
 \end{description}
\end{enumerate}
\end{jegnote}

We will also need certain subalgebras of $\qea{T}$, generated by certain subsets of the generating set for $\qea{T}$, as follows:
\[ U_{q}^{0}(\mathfrak{T}) = \lgen K_{\mu} \mid \mu \in Z \rgen \qquad U_{q}^{+}(\mathfrak{T}) = \lgen E_{i} \mid i \in I \rgen \qquad U_{q}^{-}(\mathfrak{T}) = \lgen F_{i} \mid i \in I \rgen \]
\[ \qebplus{T} = \lgen E_{i},\ K_{\mu} \mid i \in I,\ \mu \in Z \rgen \qquad
\qeb{T} = \lgen F_{i},\ K_{\mu} \mid i \in I,\ \mu \in Z \rgen 
\]

\noindent These subalgebras are the quantized enveloping algebra analogues of the Cartan subalgebra, subalgebras of positive and negative root vectors and the positive and negative Borel subalgebras, respectively.

Unfortunately, $\qea{T}$ is not a quasitriangular Hopf algebra in general.  This is because the analogue of the Drinfel\cprime d-Sklyanin quasitriangular structure for Lie bialgebras involves an infinite sum, since $\qea{T}$ is infinite-dimensional.  There are several approaches to resolving this problem.  Drinfel\cprime d (\cite{DrinfeldICM}) works in the setting of formal power series in a deformation parameter; Lusztig (\cite[Chapter 4]{LusztigBook}) introduces a topological completion.  The notion of weak quasitriangularity (Definition~\ref{s:prelims:weakqt}) was introduced by Majid in order to avoid these and remain in a purely algebraic setting.  

In the context of constructing $\qea{T}$ as a double-bosonisation starting from the Hopf algebra $U_{q}^{0}(\mathfrak{T})=k[Z]$ (the group algebra of $Z$), it follows from \cite[Proposition 18.7]{QGP} that we have a weak quasitriangular system $WQ(k[Z],k[\integ[i_{2}(I)]],\curly{R},\bar{\curly{R}})$, as follows.  Let $\{ h_{i} \mid i \in I \}$ be a basis of $k[\integ[i_{2}(I)]]$.  Then $\curly{R}(h_{i})=H_{i}$ and $\bar{\curly{R}}(h_{i})=\inv{H_{i}}$.  To extend this to the whole of $\qea{T}$, we use Lusztig's pairing between $U_{q}^{+}(\mathfrak{T})$ and $U_{q}^{-}(\mathfrak{T})$, given by $\rip{E_{i}}{F_{j}}=\inv{(\inv{q_{i}}-q_{i})}\delta_{ij}$.  This induces dual bases $\{ f^{a} \}$ and $\{ e_{a} \}$ and we have the quasi-$\curly{R}$-matrix, \ie the formal series $\sum_{a} f^{a} \tensor e_{a}$.  Then the $\curly{R}$, $\bar{\curly{R}}$ are given by appropriate evaluations against the pairing $\rip{\ }{\ }$ and this is well-defined.  We will not give explicit formul\ae\ here.

Similarly, we obtain a self-duality pairing of $\qeb{T}$, as follows.  For $i,\; j \in I$, define  \[ \rip{K_{i}}{K_{j}} = q^{\ip{i_{1}(i)}{i_{2}(j)}}, \qquad \rip{F_{i}}{F_{j}} =-\inv{(q_{i}-\inv{q_{i}})}\delta_{ij}\quad \text{and} \quad \rip{K_{i}}{F_{j}} =\rip{F_{j}}{K_{i}}=0, \] extended to the whole of $\qeb{T}\tensor \qeb{T}$.  One proof that this is a dual pairing of Hopf algebras is in \cite[Chapter~6]{Jantzen}, where the pairing is expressed as a pairing of $\qeb{T}$ with $\op{\qebplus{T}}$.  (We identify $\op{\qebplus{T}}$ with $\qeb{T}$.)  As Jantzen observes, the idea goes back to Drinfel\cprime d.  It is in order to have this pairing that we index the generators of $U_{q}^{0}(\mathfrak{T})$ by $Z$ rather than $Y$.  For the root of unity case, we refer the reader to the work of De Concini and Lyubashenko (\cite{DeConciniLyubashenko}) and the book by Brown and Goodearl (\cite{Brown-Goodearl}).

We may construct the Drinfel\cprime d double of $\qeb{T}$, $D(\qeb{T})$.  We use a variant on the original definition \cite{DrinfeldICM} suitable for infinite-dimensional Hopf algebras described in \cite[Chapter 7]{FQGT}.  This is $\qeb{T} \tensor \qebplus{T}$ with a double cross product structure given by simultaneous actions of each factor on the other, denoted $\qeb{T} \dbxsum \qebplus{T}$.  Now following Drinfel\cprime d again, we may recover $\qea{T}$ as a quotient of $D(\qeb{T})$.  Observe that $D(\qeb{T})$ is generated by $\{ F_{i} \tensor 1,\ 1\tensor E_{i} \mid i \in I \} \union \{ K_{\mu} \tensor 1,\ 1\tensor K_{\mu} \mid \mu \in Z \}$.  Then the quotient $\qea{T}$ is obtained by identifying the two Cartan parts, \ie we impose the relation $K_{\mu} \tensor 1 = 1\tensor K_{\mu}$.  The corresponding ideal defining the quotient is generated by elements of the form $K_{\mu} \tensor \inv{K}_{\mu}-1\tensor 1$.  We will refer to the projection $\mathbb{P}\colon D(\qeb{T}) \onto \qea{T}$ as Drinfel\cprime d's projection\label{s:prelims:defn:Drinproj}.

We will not consider all representations of $\qea{T}$ but as usual concentrate on those modules that decompose into weight spaces.  Among the set of weights of a module, the dominant weights are particularly important.  We extend the definition of dominant in \cite[\S 3.5.5]{LusztigBook} slightly, as we will need to consider weights and their properties with respect to more than one quantized enveloping algebra.  So we define dominance relative to certain subsets of $I$, namely those whose image in the cocharacter lattice $Y$ is a linearly independent set.  As noted by Lusztig, one can define dominance without this linear independence but it is of no use.

\begin{definition}\label{s:prelims:Sdominant} Let $\mathfrak{T} = (I,{}\Cdot{},Y,X,\ip{\ }{\ },i_{1},i_{2})$ be a root datum.  For $\lambda \in X$ and any subset $S \subseteq I$ such that the set $\{ i_{1}(s) \mid s \in S \}$ is linearly independent in $Y$, we say $\lambda$ is $S$-dominant if $\ip{i_{1}(s)}{\lambda} \in \nat$ for all $s\in S$.
\end{definition}

For any weight $\lambda$, we have two important modules, the Verma module $\Delta(\lambda)$ and the Weyl module $L(\lambda)$ of highest weight $\lambda$.  We fix $\lambda \in X$ and then, as in \cite[Section 5.5]{Jantzen}, we first define the left ideal
\[ J_{\lambda} = \sum_{i \in I} \qea{T}E_{i} + \sum_{i \in I} \qea{T}(K_{i}-q^{\ip{i_{1}(i)}{\lambda)}}). \]
The Verma module is defined as $\Delta(\lambda) \defeq \qea{T}/J_{\lambda}$ and is generated by the coset of $1$, denoted $v_{\lambda}$; $\Delta(\lambda)$ has a unique maximal submodule.  The Weyl module $L(\lambda)$ is defined to be the unique simple factor of $\Delta(\lambda)$.

\subsection{Hopf algebra gradings and split projections}\label{ss:grandsplproj}

\renewcommand{\labelenumi}{\textit{\roman{enumi}}\textup{)}\ }

The following easy proposition relates $\nat$-gradings to split projections.  We remark that we make no assumptions on the Hopf algebra structure of $H_{0}$: it need not be a group algebra, for example.

\begin{proposition}\label{ss:grandsplproj:Ngdsplproj}
Let $H=\bigdsum_{n\in \nat} H_{n}$ be an $\nat$-graded $k$-Hopf algebra.  Then $H_{0}$ is a sub-Hopf algebra of $H$.  Let $\pi\colon H \onto H_{0}$ be defined by 
\[ \pi(H_{i}) = \begin{cases} \id|_{H_{0}} & \text{if}\ i=0 \\
                              0            & \text{otherwise.}
                \end{cases} \]
Then $\pi$ is a projection of $\nat$-graded Hopf algebras, split by the inclusion $\iota\colon H_{0} \inj H$.  By this, we mean that $\pi$, $\iota$ are morphisms in the category of $k[\nat]$-comodules and are Hopf algebra maps, such that $\pi$ is surjective, $\iota$ is injective and $\pi \circ \iota = \id_{H_{0}}$ (the splitting condition).  $H_{0}$ is $\nat$-graded in the obvious way: $(H_{0})_{0}=H_{0}$, $(H_{0})_{i}=0$ $(i>0)$. \qed
\end{proposition}

Now we can use the well-known Radford--Majid theorem in the special case of $\nat$-graded Hopf algebras to see that we have both a Hopf algebra in a braided category associated to the grading and a bosonisation reconstructing our original Hopf algebra.

\begin{theorem}[\cf \cite{Radford},\cite{MajidBraidedMatrix}]\label{ss:grandsplproj:RadfordMajid}
Let $(H,H')$ be a dual pair of Hopf algebras with $H$ and $H'$ $\nat$-graded.  Assume $H'$ has an invertible antipode.  Let $H\genfrac{}{}{0pt}{}{\overset{\pi}{\onto}}{\underset{\iota}{\hookleftarrow}} H_{0}$ be the above split Hopf algebra projection.  Then there is a Hopf algebra $B$ in the braided category of $D(H_{0},H'_{0})$-modules such that $B \rtimesdot H_{0} \iso H$. \qed
\end{theorem}

\noindent Here $D$ denotes the Drinfel\cprime d double, with $H_{0}$ and $H'_{0}$ dually paired and $D(H_{0},H'_{0})=H_{0} \dbxsum \op{H'_{0}}$ (the double cross product form again).  Note that the dual pairing of $H$ and $H'$ does descend to a dual pairing of the sub-Hopf algebras $H_{0}$ and $H'_{0}$ and the (invertible) antipode of $H'$ restricts to an invertible antipode on $H'_{0}$.  

We have the following explicit descriptions of $B$ and the isomorphism, from \cite{MajidBraidedMatrix}:

\begin{enumerate}
\item $B \defeq \{ b \in H \mid b_{(1)} \tensor \pi(b_{(2)}) = b \tensor 1 \}$.  $B$ is a subalgebra of $H$, namely the subalgebra of coinvariants of $H$ under the coaction given by $\beta(h)=h_{(1)} \tensor \pi(h_{(2)})$.
\item $B$ may also be described as the image of the map $\Pi\colon H \to H$, $\Pi(h)=h_{(1)}((S\circ \iota \circ \pi)(h_{(2)}))$ for all $h\in H$.  We note that $\Pi|_{B}=\id_{B}$ and $\Pi$ is graded, since $\Pi$ is given by a composition of graded maps.
\item The action of $D(H_{0},H'_{0})$ on $B$ is given as follows. Let $b \in B$.  Then
  \begin{itemize}
  \item for $h \in H_{0}$, $h \laction b = \iota(h_{(1)})b(S\circ \iota)(h_{(2)})$ and \label{ss:grandsplproj:H0action}
  \item for $a \in H'_{0}$, $b \raction a = \ip{\pi(b_{(1)})}{a}b_{(2)}$.
  \end{itemize}
\item The braided structures on $B$ are: for $b,\: c \in B$,
  \begin{itemize}
  \item the braided coproduct $\bcoproduct b = \Pi(b_{(1)}) \tensor b_{(2)}$,
  \item the braided antipode $\bantipode b = ((\iota \circ \pi)(b_{(1)}))Sb_{(2)}$, and
  \item the braiding $\Psi=\Psi_{B,B} \in \Mor{\mathcal{D}}{B \btensor B}{B \btensor B}$, $\Psi(b \tensor c) = (\pi(b_{(1)}) \laction c) \tensor b_{(2)}$.
  \end{itemize}
\item The isomorphism $\Upsilon\colon H \to B \rtimesdot H_{0}$ is given by \begin{align*} & \Upsilon(h) =\Pi(h_{(1)}) \tensor \pi(h_{(2)}) = h_{(1)}((S \circ \iota \circ \pi)(h_{(2)})) \tensor \pi(h_{(3)}) \intertext{for all $h\in H$.  Its inverse is $\inv{\Upsilon}\colon B \rtimesdot H_{0} \to H$,} & \inv{\Upsilon}(b \tensor h) = b\cdot \iota(h) \end{align*} for $b \in B$, $h \in H_{0}$ and $\cdot$ the product in $H$.
\end{enumerate}

Recall that any Hopf algebra $H$ acts on itself by the adjoint action $\Ad_{u}(v)=u_{(1)}vSu_{(2)}$ for $u,\,v \in H$.  Furthermore, if $H$ is graded, $\mathrm{Ad}$ is a graded map.  As we saw above, $H_{0}$ acts on $B$ and indeed the formula in \textit{iii}) above may be written as $h \laction b=\Ad_{\iota(h)}(b)$ for $h\in H_{0}$ since $\iota$ is a Hopf algebra map.  In fact, $B$ is an $\mathrm{Ad}$-submodule of $H$.

Next, we note that $B$ inherits an $\nat$-grading from $H=\bigdsum_{n\in \nat} H_{n}$.  For we may define a map $\Upsilon\colon H \to H \tensor H$ by $\Upsilon(h)=h_{(1)}((S \circ \iota \circ \pi)(h_{(2)})) \tensor \pi(h_{(3)})$, the same formula as in \textit{v}) above.  For $h \in H_{n}$, we have $\Upsilon(h) \in H_{n} \tensor H_{0}$.  Now define $B_{n} = \{ b \in B \mid \Upsilon(b) \in H_{n} \tensor H_{0} \}$.  For all $n \in \nat$, $B_{n} = B \cap H_{n}$ and $B$ is an $\nat$-graded algebra: $B=\bigdsum_{n \in \nat} B_{n}$.  

Therefore we may focus our attention on the structure of the homogeneous components: $B_{n}$ is a $D(H_{0},H'_{0})$-submodule of $B$ and $B_{0} = k$.  We note that this tells us that $B$ satisfies the defining Nichols algebra condition (\hspace{-1pt}\textit{a}\hspace{-0.5pt}) of Definition~\ref{s:prelims:d:Nicholsalg}.  Also it is well-known that $B_{0}=k$ implies that $B_{1} \subseteq \brprim{B}$: one uses the fact that the braided coproduct $\bcoproduct$ is a graded map.  However one does not know in general whether condition (\hspace{-1pt}\textit{b}\hspace{-0.5pt}) holds, \ie whether $B_{1}=\brprim{B}$.

\section{Sub-root data and their associated braided Hopf algebras and triangular decompositions}\label{s:defnofB}

We begin by defining sub-root data, denoted $\mathfrak{J} \subseteq_{\iota} \mathfrak{T}$, an abstraction of the Lie algebra-subalgebra pairs we considered in \cite{LieInduction}.  Our reason for introducing these is that any choice of sub-root datum $\mathfrak{J} \subseteq_{\iota} \mathfrak{T}$ gives rise to an $\nat$-grading of the quantum negative Borel subalgebra $\qeb{T}$.  We analyse the structure of the zeroth homo\-geneous component of this grading, showing that it is a semi-direct product of $\qeb{J}$, the quantum negative Borel subalgebra associated to $\mathfrak{J}$, by a group Hopf algebra.  This allows us to show that $\qea{T}$ may be expressed as a double-bosonisation of a similar semi-direct product $\widetilde{\qea{J}}$ by an $\nat$-graded Hopf algebra $B=B(\mathfrak{T},\mathfrak{J},\iota,q)$ in the braided category of $\widetilde{\qea{J}}$-modules.  

\subsection{Sub-root data}\label{s:defnofB:ss:subrootdata}

We define our principal object of study, a pair of suitably related root data.  

\begin{definition}\label{s:defnofB:subrootdatum}
Let 
\begin{align*} \mathfrak{T} & =(I,{}\Cdot{},Y,X,\ip{\ }{\ },i_{1}\colon I\inj Y,i_{2}\colon I\inj X) \\
               \mathfrak{J} & =(J,{}\Cdotp{},Y',X',\ip{\ }{\ }',i_{1}'\colon J\inj Y', i_{2}'\colon J\inj X')
\end{align*} 
be two root data.  Then we say $\mathfrak{J}$ is a sub-root datum of $\mathfrak{T}$ via $\mathbf{\iota}$ if 
\begin{enumerate}
\setlength{\itemsep}{0.5pt}
\item $\iota\colon J\inj I$ is injective,
\item the restriction of ${}\Cdot{}$ to the subgroup $\integ[\iota(J)] \subseteq \integ[I]$ is $\,{}\Cdotp{}$,
\item there exist injective group homomorphisms $s_{Y}\colon Y' \inj Y$, $s_{X}\colon X' \inj X$, such that $Y/s_{Y}(Y')$ and $X/s_{X}(X')$ are free Abelian,
\item the restriction of $\ip{\ }{\ }$ to the subgroup $s_{Y}(Y') \cross s_{X}(X') \subseteq Y\cross X$ is $\ip{\ }{\ }'$,
\item there exists a subgroup $X''$ of $X$ such that $X=X'\dsum X''$ and $s_{Y}(Y')$ is orthogonal to $X''$, \ie $\ip{s_{Y}(y')}{x''}=0$ for all $y' \in Y'$, $x'' \in X''$ and
\item $s_{Y} \circ i_{1}'=i_{1} \circ \iota$ and $s_{X} \circ i_{2}'=i_{2} \circ \iota$.
\end{enumerate}
We will denote this by $\mathfrak{J} \subseteq_{\iota} \mathfrak{T}$.
\end{definition}

\pagebreak
\begin{jegnote}{s}
\begin{enumerate}
\item The maps $s_{Y}$, $s_{X}$ will be suppressed in what follows: we think of $Y'$ and $X'$ as subgroups of $Y$ and $X$ respectively, identifying $Y'$ (resp.\ $X'$) with its image under $s_{Y}$ (resp.\ $s_{X}$).
\item We note that if $G/G'$ is a free Abelian quotient of an Abelian group $G$, such as posited in \textit{iii}), then $G=G' \dsum G''$ for some subgroup $G''$ of $G$ (see, for example, \cite[Section 4.2]{Robinson}).  Therefore the condition in \textit{v}) is concerned principally with the inner product, rather than the existence of $X''$.  The splitting $X=X'\dsum X''$ will be considered as fixed by the choice of sub-root datum.
\item We must specify the map $\iota$, rather than just a set inclusion $J\subseteq I$.  For example, we distinguish between the two embeddings $\iota_{1}(m)=m$ and $\iota_{2}(m)=l-m+1$ of the subset $J=\{ 1,\ldots , l-1 \}$ in $I=\{ 1,\ldots , l\}$.
\end{enumerate}
\end{jegnote}

\noindent A sub-root datum gives rise to an algebra-subalgebra pair of quantized enveloping algebras, in the obvious way.

\begin{lemma}
Let $\mathfrak{J} \subseteq_{\iota} \mathfrak{T}$ be a sub-root datum.  There is an injective Hopf algebra homomorphism $\iota\colon \qea{J} \to \qea{T}$, defined on the generators of $\qea{J}$ by $\iota(E_{j})=E_{\iota(j)}$, $\iota(F_{j})=F_{\iota(j)}$ and $\iota(K_{i'_{1}(j)})=K_{i_{1}(\iota(j))}$ for all $j\in J$. \qed
\end{lemma}

The conditions \textit{iv}) and \textit{vi}) of the definition of a sub-root datum ensure that the relations are respected.  As an example, take the sub-root datum $A_{2} \subseteq_{\iota} A_{3}$ with $\iota\colon \{1,2\} \to \{1,2,3\}$, $\iota(j)=j$.

We may build up root data by taking direct sums.

\begin{definition}\label{s:defnofB:dsumofrd}
Let 
\begin{align*} \mathfrak{T} & =(I,{}\Cdot{},Y,X,\ip{\ }{\ },i_{1}\colon I\inj Y,i_{2}\colon I\inj X) \\
               \mathfrak{J} & =(J,{}\Cdotp{},Y',X',\ip{\ }{\ }',i_{1}'\colon J\inj Y', i_{2}'\colon J\inj X')
\end{align*} 
be two root data.  Then the direct sum\/ $\mathfrak{T} \dsum \mathfrak{J}$ of\/ $\mathfrak{T}$ and\/ $\mathfrak{J}$ is the root datum with underlying set $I \union J$, symmetric bilinear form\/ $\cdot_{\dsum}=\cdot \dsum \cdot'$, associated finitely generated free Abelian groups\/ $Y\dsum Y'$ and $X\dsum X'$, non-degenerate bilinear form $\ip{\ }{\ }_{\dsum}\colon (Y\dsum Y') \cross (X\dsum X') \to \integ$ defined by $\ip{y_{1}\dsum y_{2}}{x_{1}\dsum x_{2}}_{\dsum}=\ip{y_{1}}{x_{1}}+\ip{y_{2}}{x_{2}}'$ and associated inclusions $i_{r}\dsum i_{r}'\colon I \union J \to Y\dsum Y'$, $r=1,2$, with $(i_{1}\dsum i_{1}') |_{I}=i_{1}$, etc.
\end{definition}

\noindent It is clear that this is again a root datum.  The notions of sub-root datum and direct sum are suitably compatible: $\mathfrak{T}$, $\mathfrak{J}$ are sub-root data of $\mathfrak{T} \dsum \mathfrak{J}$ via the inclusions $I, J \subseteq I \union J$.

Let $\mathfrak{J} \subseteq_{\iota} \mathfrak{T}$ be a sub-root datum of $\mathfrak{T}$ via $\iota$.  

\begin{definition}\label{s:defnofB:rhodefn} We have a splitting $X = X' \dsum X''$ so let $\pi\colon X \to X/X''$ be the canonical projection and $i\colon X/X'' \to X'$ the isomorphism of $X/X''$ with $X'$.  Define the restriction map $\rho \colon X \to X'$ to be $\rho = i \circ \pi$.  In particular, we have $\rho |_{X'} = \id_{X'}$.  If $\lambda \in X$, we will often denote $\rho(\lambda) \in X'$ by $\lambda'$.  This is consistent with the decomposition $\lambda = \lambda' \dsum \lambda''$, $\lambda' \in X'$, $\lambda'' \in X''$ given by $X = X' \dsum X''$.
\end{definition}

Note that for all $\mu' \in Y'$, we have $\ip{ \mu'}{\rho(\lambda)}'=\ip{\mu'}{\lambda}$.  We call $\rho$ the restriction map as it encodes the restriction of weight representations from $\qea{T}$ to $\qea{J}$.  For let $M = \bigdsum_{\lambda \in X} M^{\lambda}_{\mathfrak{T}}$ be a weight module for $\qea{T}$.  Then $M$ is a weight module for $\qea{J}$ by restriction, so we may write $M = \bigdsum_{\lambda' \in X'} M^{\lambda'}_{\mathfrak{J}}$.  Furthermore, \[ M^{\lambda'}_{\mathfrak{J}} = \bigdsum_{\substack{\lambda \in X,\\ \rho(\lambda)=\lambda'}} M^{\lambda}_{\mathfrak{T}} = \bigdsum_{\lambda'' \in X''} M^{\lambda' + \lambda''}_{\mathfrak{T}}. \]

\noindent It is then natural to ask if $\rho$ preserves dominance (Definition~\ref{s:prelims:Sdominant}).  We say a root datum $\mathfrak{T}$ with associated embedding $i_{1}\colon I \inj Y$ is $Y$-regular if the set $\Image{i_{1}}$ is linearly independent in $Y$.  (We may define $X$-regularity in a similar fashion.  Lusztig (\cite[\S 6.3.3]{LusztigBook}) notes that if $\mathfrak{T}$ is a finite type root datum then $\mathfrak{T}$ is automatically both $X$- and $Y$-regular.)  Clearly, if $\mathfrak{T}$ is $Y$-regular then any sub-root datum of $\mathfrak{T}$ is too.  Furthermore, dominance is preserved under $\rho$.

\subsection{The quantum negative Borel subalgebra $\qeb{T}$}\label{s:defnofB:ss:HAqebI}

Let $\mathfrak{J} \subseteq_{\iota} \mathfrak{T}$.  Then $\qea{J}$ may be identified with the sub-Hopf algebra of $\qea{T}$ with generators $E_{j}$, $F_{j}$, $j\in \iota(J)$, $K_{\nu}$, $\nu \in Z' \defeq \integ[i'_{1}(J)]$.  

Recall that $\qea{T}$ has a $\integ[I]$-grading, given by $\deg E_{i} = -\deg F_{i} = i$, $\deg K_{\mu} = 0$.  However, it also has many $\integ$-gradings.  Let $\gamma\colon I\to \integ$ be any function.  Then $\qea{T}$ is $\integ$-graded by $\deg E_{i} = -\deg F_{i} = \gamma(i)$, $\deg K_{\mu} = 0$.  We see this by noting that all the defining relations are homogeneous in degree (\cf \cite[Section~1.5]{Kac}).  In particular, $\qea{T}$ has a $\integ$-grading associated to any sub-root datum $\mathfrak{J} \subseteq_{\iota} \mathfrak{T}$.  Let $D=I\setminus \iota(J)$ and let $\chi_{D}\colon I \to \{ 0,1 \}$ be the indicator function for $D$, \ie 
\[ \chi_{D}(i)=\begin{cases} 1 & \text{if}\ i\in D \\
                             0 & \text{if}\ i\not\in D
               \end{cases} \]
Then, as above, regarding $\chi_{D}$ as a function $I \to \integ$ we have a $\integ$-grading on $\qea{T}$:  
\[ \qea{T}=\bigdsum_{n\in \integ} \qea{T}_{[n]}. \]
In particular, $\qea{J} \subseteq \qea{T}_{[0]}$ and $U_{q}^{0}(\Lie{T})=\lgen K_{\mu} \mid \mu \in Z \rgen \subseteq \qea{T}_{[0]}$.

Consider now the sub-Hopf algebra $\qeb{T}$ of $\qea{T}$, the analogue of the negative Borel subalgebra, generated by the set $\{ F_{i} \mid i\in I \} \union \{ K_{\mu} \mid \mu \in Z \}$.  Then $\qeb{T}$ is $\nat$-graded via $\chi_{D}$: $\deg F_{i} = \chi_{D}(i)$, $\deg K_{\mu} = 0$.  In particular, $\qebO{T}$ contains $\qeb{J}$, which is generated by $\{ F_{j} \mid j\in \iota(J) \} \union \{ K_{\nu} \mid \nu \in Z' \}$.  Note, though, that $\qeb{T} \neq \bigdsum_{i\leq 0} \qea{T}_{[i]}$ since for example for any $i\in I$, $E_{i}F_{i} \in \qea{T}_{[0]}$ but $E_{i}F_{i} \notin \qeb{T}$.  Also, as we recalled in Subsection~\ref{s:prelims:ss:rootdata+QEAs}, $\qeb{T}$ is self-dually paired.  Indeed $(\qeb{T},\qeb{T})$ is a dual pair of $\nat$-graded Hopf algebras.  Hence, Proposition~\ref{ss:grandsplproj:Ngdsplproj} and Theorem~\ref{ss:grandsplproj:RadfordMajid} apply to $\qeb{T}$ and we have the following.

\begin{theorem}\label{s:defnofB:Bexists}
Let $\mathfrak{J} \subseteq_{\iota} \mathfrak{T}$ be a sub-root datum of $\mathfrak{T}$ and let $\qeb{T}=\dsum_{n\in \nat} \qeb{T}_{[n]}$ be the associated $\nat$-graded sub-Hopf algebra of $\qea{T}$.  Then there exists a Hopf algebra $B=B(\mathfrak{T},\mathfrak{J},\iota,q)$ in the braided category of $D(\qebO{T})$-modules such that $\qeb{T} \iso B \rtimesdot \qebO{T}$. \qed   
\end{theorem}

\noindent Here we have $D(\qebO{T})=\qebO{T} \dbxsum \qebplus{T}_{[0]}$.  We now examine in more detail the structure of $\qebO{T}$.  We see immediately that the zeroth graded component $\qebO{T}$ of $\qeb{T}$ is generated by the set $\{ F_{j} \mid j\in \iota(J) \} \union \{ K_{\mu} \mid \mu \in Z \}$.  As noted above, $\qeb{J} \subseteq \qebO{T}$ and indeed is a sub-Hopf algebra.  We show that $\qebO{T}$ is a semi-direct (or smash) product of $\qeb{J}$ by $k[Z/Z']$.  

Note that $\integ[i'_{1}(J)] = Z' \subseteq Z=\integ[i_{1}(I)]$ and the quotient $Z/Z'$ is free Abelian---the quotient may be identified with $\integ[i_{1}(D)]$ where $D=I\setminus \iota(J)$.  Then since $Z/Z'$ is a free Abelian quotient of a free Abelian group $Z$, we have $Z=Z' \dsum Z''$ for some subgroup $Z''$ of $Z$ ($Z''$ is isomorphic to $\integ[i_{1}(D)]$).

\begin{proposition}\label{s:defnofB:extn}
$\qebO{T}\iso \qeb{J} \rtimes k[Z/Z']$ as Hopf algebras, where $k[Z/Z']$ is the group Hopf algebra of $Z/Z'$.
\end{proposition}

\begin{proof} The splitting $Z=Z' \dsum Z''$ yields a unique decomposition of elements of $Z$ into elements of $Z'$ and $Z''$: for $\mu \in Z$, we have $\mu = \mu' \dsum \mu''$ for (unique) $\mu' \in Z'$, $\mu'' \in Z''$.  Therefore to each $\mu \in Z$ we have a unique associated pair $(\mu',\nu)$ with $\mu' \in Z'$, $\nu = \hat{\pi}(\mu'') \in Z/Z'$.  Define $p_{1}(\mu)=\mu'$, $p_{2}(\mu)=\nu$ for all $\mu \in Z$.  We also set $q_{1}$ and $q_{2}$ to be the natural inclusions of the subgroups $Z'$ and $Z''$ into $Z\dsum Z''$, respectively, and denote by $r:Z/Z' \to Z''$ the natural isomorphism.

The algebra $\qeb{J}$ is generated by $\{ F_{j},K_{\alpha} \mid j\in \iota(J), \alpha \in Z' \}$; let $k[Z/Z']$ have generating set $\{ L_{\beta} \mid  \beta \in Z/Z' \}$.  Then we may form a semi-direct product of algebras of $\qeb{J}$ by $k[Z/Z']$ by the action $L_{\beta} \laction K_{\alpha}=K_{\alpha}$ and $L_{\beta} \laction F_{j}=q^{\ip{(q_{2} \circ r)(\beta)}{i_{2}(j)}}F_{j}$, extended linearly and to products.  (That is, $k[Z/Z']$ acts trivially on the subalgebra generated by the $K_{\alpha}$ and by the same scalar as in \textbf{(R3)} on generators of $\qeaminus{J}$, as $K_{(q_{2} \circ r)(\beta)}$ does in the adjoint action.)

It is straightforward to see that $\qeb{J} \rtimes k[Z/Z']$ is isomorphic to $\qebO{T}$ as algebras---the definition of the action yields the correct relations---and indeed as Hopf algebras, taking the former with the tensor product coalgebra structure and antipode, again from the definitions of these structures on $\qea{I}$. 
\end{proof}

\subsection{$\qea{T}$ is a double-bosonisation}\label{s:defnofB:ss:uqIisaDB}

Recall from Theorem~\ref{s:defnofB:Bexists} that we constructed $B=B(\mathfrak{T},\mathfrak{J},\iota,q)$ in the (braided) category of $D(\qebO{T})$-modules.  However, to reconstruct $\qea{T}$ as a double-bosonisation, we require $B$ in the category of modules for the algebra that is built from the quantized enveloping algebra $\qea{J}$ associated to the sub-system and the group Hopf algebra $k[Z/Z']$, not just the ``half'' $\qeb{J}\rtimes k[Z/Z']$.  To see that this is indeed the case, we make use of our analysis of the structure of $\qebO{T}$ and define a projection from the double $D(\qebO{T})$ to this algebra whose kernel annihilates $B$.

Let $\widetilde{\qea{J}}\defeq \qea{J} \rtimes k[Z/Z']$ be the semi-direct product of algebras given by extending the action of $k[Z/Z']$ on generators of $\qeb{J}$ described previously to those of $\qea{J}$ by additionally setting $L_{\beta} \laction E_{j}=q^{-\ip{(q_{2}\circ r)(\beta)}{i_{2}(j)}}E_{j}$ for $j\in \iota(J)$.  We give $\widetilde{\qea{J}}$ the tensor product Hopf algebra structure, as before.

\begin{lemma}\label{s:defnofB:UqJextquotofD}
The Hopf algebra $\widetilde{\qea{J}}$ is a quotient Hopf algebra of $D(\qebO{T})$ and the kernel of the corresponding natural projection annihilates $B$.
\end{lemma}

\begin{proof}
The double $D(\qebO{T})$ is generated by $\{ F_{j} \tensor 1,\ 1\tensor E_{j},\ K_{\mu} \tensor 1,\ 1\tensor K_{\mu} \mid j \in \iota(J),\  \mu \in Z \}$ and we may define $\Phi\colon D(\qebO{T}) \onto \widetilde{\qea{J}}$ by 
\begin{align*} 
\Phi(F_{j} \tensor 1) & = F_{j} \tensor 1, &
\Phi(1 \tensor E_{j}) & =  E_{j} \tensor 1, \\
\Phi(K_{\mu} \tensor 1) & = K_{p_{1}(\mu)} \tensor L_{p_{2}(\mu)}, &
\Phi(1 \tensor K_{\mu}) & =  K_{p_{1}(\mu)} \tensor L_{p_{2}(\mu)}, 
\end{align*}
extended linearly and multiplicatively, with $p_{1}$, $p_{2}$ as above.  It is easily verified that this is a Hopf algebra projection, as the action of $k[Z/Z']$ on $\qea{J}$ reproduces the appropriate commutation relations.
\end{proof}

The kernel of this map $\Phi$ is clearly generated by $\{ K_{\mu} \tensor \inv{K}_{\mu} - 1\tensor 1 \mid \mu \in Z \}$, since as for Drinfel\cprime d's projection $\mathbb{P}$ we identify $K_{\mu} \tensor 1$ and $1 \tensor K_{\mu}$ in the image.  The kernel of $\Phi$ annihilates $B$, since the identified elements in the quotient $K_{\mu} \tensor 1$ and $1 \tensor K_{\mu}$ have equal (left) actions on $B$.  Hence $B$ is a $\widetilde{\qea{J}}$-module.  Indeed, unwinding the double and bosonisation formul\ae, one sees that the action is precisely the restriction of the adjoint action of $\qea{T}$ on itself.  Consequently by using this action we may construct $B\rtimesdot \widetilde{\qea{J}}$, which has $B\rtimesdot \qebO{T}\iso \qeb{T}$ as a subalgebra.

We conclude by showing that $\qea{T}$ is isomorphic to the double-bosonisation of $B$ and its dual by $\widetilde{\qea{J}}$.  We carried out the above analysis on $\qeb{T}$, to obtain a braided Hopf algebra $B$ such that there is an isomorphism $\beta_{\scriptscriptstyle \leqslant}\colon \qeb{T} \stackrel{\iso}{\too} B \rtimesdot \qebO{T}$.  However, we could equally well start with the self-dual Hopf algebra $\qebplus{T}$ and obtain a braided Hopf algebra $\op{(B')}$ in the braided category of right $D(\qebplus{T}_{[0]})$-modules such that $\beta_{\scriptscriptstyle \geqslant}\colon \qebplus{T} \stackrel{\iso}{\too} \qebplus{T}_{[0]} \ltimesdot \op{(B')}$.  Furthermore, $B$ and $B'$ are dually paired braided Hopf algebras, via Lusztig's pairing.

\begin{theorem}\label{s:defnofB:isotoDB}
Let $\mathfrak{J} \subseteq_{\iota} \mathfrak{T}$ be a sub-root datum of $\mathfrak{T}$.  Then 
\[ \qea{T} \iso B \rtimesdot \widetilde{\qea{J}} \ltimesdot \op{(B')} \]
as Hopf algebras.
\end{theorem}

\begin{proof}
The stated double-bosonisation is well-defined, as $\qea{T}$ has an associated weak quasitriangular system (see Subsection~\ref{s:prelims:ss:rootdata+QEAs}) and this restricts to $\widetilde{\qea{J}}$.  The cross-relation in the double-bosonisation is the quantized enveloping algebra defining relation \textbf{(R4)}, the commutation relation for $E_{i}$ and $F_{j}$---this relation is also encoded in the cross-relations of the double (see for example \cite[Section 3.2]{Joseph} or \cite[Example 18.8]{QGP}).

We may construct the double $D(B \rtimesdot \widetilde{\qea{J}})=(B \rtimesdot \widetilde{\qea{J}}) \dbxsum (\widetilde{\qea{J}} \ltimesdot \op{(B')})$ and the quotient obtained by identifying the two copies of $\widetilde{\qea{J}}$ is precisely the stated double-bosonisation, by \cite[Theorem~6.2]{DoubleBos}.  (Majid's result demonstrates that a double-bosonisation $B\rtimesdot H \ltimesdot B'$ may be constructed as a quotient of the double of $B\rtimesdot H$ in precisely this way.)

However, $B\rtimesdot \widetilde{\qea{J}}$ and $\widetilde{\qea{J}}\ltimesdot \op{(B')}$ are easily seen to be isomorphic to $\qeb{T}\qea{J}$ and $\qea{J}\qebplus{T}$ respectively.  Hence we see that this quotient of $D(B\rtimesdot \widetilde{\qea{J}}\iso D(\qeb{T}\qea{J})$ is isomorphic to $\qea{T}$, by the same argument as that for Drinfel\cprime d's projection $\mathbb{P}$.  So the double-bosonisation is isomorphic to the full quantized enveloping algebra.
\end{proof}

\vspace{1em}
\begin{example*} We conclude that we have $B=B(A_{3},A_{2},\iota,q)$ a Hopf algebra in the braided category of $\widetilde{\qea{sl_{\mathrm{3}}}}$-modules and its dual $B'$ such that
\[ \qea{sl_{\mathrm{4}}} \iso B \rtimesdot \widetilde{\qea{sl_{\mathrm{3}}}} \ltimesdot \op{(B')}. \]
We will describe $B$ explicitly at the end of Section~\ref{s:Bstr}.
\end{example*}

\vspace{1em}
\begin{example*} Recall that we have defined the direct sum $\mathfrak{T} \dsum \mathfrak{J}$ of two root data $\mathfrak{T}$, $\mathfrak{J}$ (Definition~\ref{s:defnofB:dsumofrd}) and $\mathfrak{T}$ is a sub-root datum of $\mathfrak{T} \dsum \mathfrak{J}$.  Now $\qea{T \dsum J} \iso \qea{T} \tensor \qea{J}$ since if $i\in I$ and $j\in J$, $C_{ij}=0$ and so $E_{i}E_{j}=E_{j}E_{i}$ and $F_{i}F_{j}=F_{j}F_{i}$.

Since $\mathfrak{T} \subseteq_{\iota} \mathfrak{T} \dsum \mathfrak{J}$, by the preceding Theorem, we obtain $B=B(\mathfrak{T} \dsum \mathfrak{J},\mathfrak{T},\iota,q)$ so that $\qea{\mathfrak{T} \dsum \mathfrak{J}} \iso B \rtimesdot \widetilde{\qea{\mathfrak{T}}} \ltimesdot \op{(B')}$.  We see that $\widetilde{\qea{T}}=\qea{T} \tensor U_{q}^{0}(\mathfrak{J})$ and $B=\qeaminus{J}$, $\op{(B')}=\qeaplus{J}$.  Then 
\begin{eqnarray*} \qea{\mathfrak{T} \dsum \mathfrak{J}} & \iso & B \rtimesdot (\qea{T} \tensor U_{q}^{0}(\mathfrak{J})) \ltimesdot \op{(B')} \\
 & \iso & \qea{T} \tensor (\qeaminus{J} \rtimesdot U_{q}^{0}(\mathfrak{J}) \ltimesdot \qeaplus{J} ) \\
 & \iso & \qea{T} \tensor \qea{J}.
\end{eqnarray*}
So the construction is compatible with direct sums.
\end{example*}
\vspace{1em}

\section{The structure of $B$}\label{s:Bstr}

From our results on general braided Hopf algebras $B$ arising from split projections of graded Hopf algebras, we know that $B=B(\mathfrak{T},\mathfrak{J},\iota,q)$ associated to $\mathfrak{J} \subseteq_{\iota} \mathfrak{T}$ is a graded braided Hopf algebra and an $\Ad$-submodule of $\qeb{T}$.  We now examine the module, algebra and braided-coalgebra structures of $B$ further.

We analyse the algebra structure of $B$, giving a set of generators.  In particular, these generators all have degree one.  We also examine the module structure of $B$ and see that $B_{1}$ is a direct sum of (possibly quotients of) Weyl modules and that the higher graded components are sums of submodules of tensor products of these.  Finally, we observe that the graded dual of $B$ is also generated in degree one and hence $B$ is a Nichols algebra.

\subsubsection*{Notation}  

For $S$ a finite set, denote by $S^{\nat}$ the set of all finite sequences of elements of $S$, including the empty sequence, $\emptyset$.  If $\alpha \in S^{\nat}$, $l(\alpha)$ will denote the length of $\alpha$; $l(\emptyset)=0$.  If $i\colon S \inj M$ is an injective map from $S$ into a commutative monoid $M$, we define the weight of $\alpha \in S^{\nat}$ with respect to $i$ to be $\weight{i}{\alpha}=\sum_{j=1}^{l(\alpha)} i(\alpha_{j})$.  We set $\weight{i}{\emptyset}=0$ (the identity element of $M$).  For $i_{1}\colon I \inj Y$, we will write $\weight{1}{\ }$ for $\weight{i_{1}}{\ }$ and similarly for $i_{2}\colon I \inj X$.  For $\gamma=(\gamma_{1},\dotsc ,\gamma_{l}) \in S^{\nat}$ set $F_{\gamma} \defeq F_{\gamma_{1}}\dotsm F_{\gamma_{l}}$.

\subsection{The algebra structure of $B$}\label{s:Bstr:ss:Balgstr}

From the general results, we know that $B=B(\mathfrak{T},\mathfrak{J},\iota,q)$ is a graded algebra; however the general results do not give us much more information about $B$ than this.  Since $\qea{T}$ is defined by generators and relations, we would also like to have a presentation for $B$.  As a first step, we may explicitly identify a set of generators of $B$, as follows.

\begin{theorem}\label{s:Bstr:alggensofB}  Let $A$ be the $\widetilde{\qea{J}}$-submodule of $B$ generated by the set $\{ F_{\gamma}H_{\weight{1}{\gamma}} \mid \gamma \in D^{\nat} \}$ and let $\curly{A}$ be the subalgebra of $B$ generated by $A$.  Then $\curly{A}=B$.
\end{theorem}

\begin{proof} Recall from Section~\ref{ss:grandsplproj} that we have an isomorphism $\Upsilon \colon \qeb{T} \to B \rtimesdot \qebO{T}$, $\Upsilon(h)=h_{(1)}((S \circ \iota \circ \pi)(h_{(2)})) \tensor h_{(3)}$.  We calculate $\Upsilon$ on the generators of $\qeb{T}$ and obtain
\begin{align*} \Upsilon(F_{i}) & = \begin{cases} 1 \tensor F_{i} & \text{if}\ i\in \iota(J) \\ F_{i}H_{i} \tensor \inv{H_{i}} & \text{if}\ i\in D \end{cases} \\
\Upsilon(K_{\mu}) & = 1 \tensor K_{\mu} \quad \forall\, \mu \in Z. \end{align*}

\noindent For $\emptyset \in D^{\nat}$, $F_{\emptyset}H_{\weight{1}{\emptyset}}=H_{0}=1$ (by convention).  

We wish to show that $\Upsilon(\qeb{T}) \subseteq \curly{A} \tensor \qebO{T}$.  Consider a monomial $F_{\alpha}K_{\mu}$, $\alpha \in I^{\nat}$, $\mu \in Z$.  Recall that monomials of this form are a basis for $\qeb{T}$.  Then $\Upsilon(F_{\alpha}K_{\mu})=\Upsilon(F_{\alpha})(1 \tensor K_{\mu})$ and so we need only show that $\Upsilon(F_{\alpha}) \in \curly{A} \tensor \qebO{T}$.

We proceed by induction on $l(\alpha)$.  For $l(\alpha)=1$, the above formul\ae\ for $\Upsilon(F_{i})$ suffice.  Assume now that $\Upsilon(F_{\alpha}) \in \curly{A} \tensor \qebO{T}$ for all $\alpha \in I^{\nat}$ with $l(\alpha)=r$, for some $r$.  Let $\beta \in I^{\nat}$ with $l(\beta)=r+1$.  Then we may write $F_{\beta}=F_{\alpha}F_{i}$ with $\alpha \in I^{\nat}$, $l(\alpha)=r$ and $i\in I$.  Write $\Upsilon(F_{\alpha})=x^{(1)} \tensor x^{(2)}$ in Sweedler notation, with $x^{(1)} \in \curly{A}$ and $x^{(2)} \in \qebO{T}$ by the inductive hypothesis.  

We then have two cases:
\begin{enumerate}
\item if $i \in \iota(J)$ then {\allowdisplaybreaks \begin{align*} \Upsilon(F_{\beta}) & = \Upsilon(F_{\alpha})\Upsilon(F_{i}) \\ & = (x^{(1)} \tensor x^{(2)})(1 \tensor F_{i}) \\ & = x^{(1)} \tensor x^{(2)}F_{i} \\ & \in \curly{A} \tensor \qebO{T}. \end{align*} }
\item if $i \in D$ then {\allowdisplaybreaks \begin{align*} \Upsilon(F_{\beta}) & = \Upsilon(F_{\alpha})\Upsilon(F_{i}) \\ & = (x^{(1)} \tensor x^{(2)})(F_{i}H_{i} \tensor \inv{H_{i}}) \\ & = x^{(1)}\Ad_{\iota({x^{(2)}}_{(1)})}(F_{i}H_{i}) \tensor {x^{(2)}}_{(2)}\inv{H_{i}} \\ & \qquad \qquad \qquad \text{(by the form of the product in $B \rtimesdot \qebO{T}$)} \\ & \in \curly{A} \tensor \qebO{T} \end{align*} }
\end{enumerate}
Note that $\Ad_{\iota({x^{(2)}}_{(1)})}(F_{i}H_{i}) \in \curly{A}$ since $x^{(2)} \in \qebO{T}$ and $\qebO{T}$ is a sub-Hopf algebra of $\qeb{T}$.  Thus, $\Upsilon(\qeb{T}) \subseteq \curly{A} \tensor \qebO{T}$ but then $\Upsilon(B) = B\tensor 1 \subseteq \curly{A} \tensor 1$ and hence $B=\curly{A}$.
\end{proof}

From this, the following is immediate.

\begin{corollary}\label{s:Bstr:B1generatesB} The submodule $B_{1}$, which is the first graded component of $B$, is generated as a $\widetilde{\qea{J}}$-module by the set $\{ F_{d}H_{d} \mid d\in D \}$. Furthermore $B$ is generated as an algebra by $\widetilde{B_{1}}=B_{1} \dsum B_{0} = B_{1} \dsum k1$.
\end{corollary}

\begin{proof}  This follows from the proof of the theorem---in particular, part \textit{ii)} (the case $i \in D$) and the fact that $B_{0}=k$.
\end{proof}

\subsection{The module structure of $B$}\label{s:Bstr:ss:Bmodstr}

We would like some additional information on the module structure of $B$, in particular regarding its set of weights.  Recall that $B$ is an $\mathrm{Ad}$-submodule of $\qeb{T}$.  Although we want to know the module structure of $B$ as a $\widetilde{\qea{J}}$-module, we first consider the adjoint action of $\qea{T}$ on $\qeb{T}$.  For $\alpha \in I^{\nat}$, $\mu \in Z$, the weight of $F_{\alpha}K_{\mu}$ for the adjoint action is $\weight{2}{-\alpha} = -\sum_{j=1}^{l(\alpha)} i_{2}(\alpha_{j})$, where $i_{2}\colon I \inj X$ is the injection of the index set $I$ into the character lattice $X$.  Since the $F_{\alpha}K_{\mu}$ span $\qeb{T}$, the set of weights of $\qeb{T}$ for $\mathrm{Ad}$ is $-\nat[i_{2}(I)]$.  Define $\mult{D}{\alpha}=|\{ \alpha_{j} \mid j \in D \} |$ for $\alpha \in I^{\nat}$ and $\mult{D}{\weight{2}{-\alpha}} \defeq \mult{D}{\alpha}$.  Note that $\mult{D}{\alpha}=\deg F_{\alpha}$ for the grading described at the start of Subsection~\ref{s:defnofB:ss:HAqebI}.

We have $F_{d}H_{d} \in B$ for $d \in D$; by the above, $F_{d}H_{d}$ has weight $\weight{2}{-d}$.  Now $\weight{2}{-d}$ is not in general $I$-dominant (see Definition~\ref{s:prelims:Sdominant}): $\ip{i_{1}(d)}{-i_{2}(d)}=-C_{dd}=-2$.  However, its image under $\rho$ is $J$-dominant.

Next we consider primitive vectors for the action of $\qea{J}$, that is, $b\in B$ such that $E_{j} \laction b = 0$ for all $j\in J$.  The elements $F_{\gamma}H_{\weight{1}{\gamma}}$, $\gamma \in D^{\nat}$, which have already appeared in Theorem~\ref{s:Bstr:alggensofB} are primitive vectors of weight $\rho(-\weight{1}{\gamma})$, since $E_{\iota(j)} \laction F_{\gamma}H_{\weight{1}{\gamma}}=\Lbracket{E_{\iota(j)}}{F_{\gamma}}H_{\weight{1}{\gamma}}= 0$ by \textbf{(R4)}: $\gamma_{i} \neq \iota(j)$ for all $i$.

We have that $B_{n}$ is a $\qea{J}$-submodule of $B$ for all $n\in \nat$.  Since $B_{0}=k$, $B_{0}$ is the trivial $\qea{J}$-module. By Corollary~\ref{s:Bstr:B1generatesB}, $B_{1}$ is generated as a $\widetilde{\qea{J}}$-module by its primitive vectors, namely the set $\{ F_{d}H_{d} \mid d\in D \}$.  Let $V(\lambda_{d}')$ be the submodule of $B_{1}$ generated by $F_{d}H_{d}$.  We remark that although $V(\lambda_{d_{1}}') \intersection V(\lambda_{d_{2}}') = 0$ for $d_{1} \neq d_{2}$, we may have $V(\lambda_{d_{1}}') \iso V(\lambda_{d_{2}}')$ as $\widetilde{\qea{J}}$-modules.  Indeed, possibly $\lambda_{d_{1}}'=\lambda_{d_{2}}'$.  Note also that $B_{1}$ is a direct sum of finitely many submodules $V(\lambda_{d}')$, since $D$ is finite.

A straightforward application of the universal property of Verma modules and the method of the proof of \cite[Proposition 3.5.8]{LusztigBook} yields the following.

\begin{proposition}\label{s:Bstr:B1integ} For all $d \in D$, $V(\lambda_{d}')$ is integrable, therefore $B_{1}$ is integrable.  If $L(\lambda_{d}')$ is finite-dimensional for all $d\in D$ then $V(\lambda_{d}')$ and $B_{1}$ are finite-dimensional.  \qed
\end{proposition}

By Theorem~\ref{s:Bstr:alggensofB} the submodules $B_{n}$, $n\geq 2$, are direct sums of submodules of tensor products of the $V(\lambda_{d}')$.  Hence we may deduce the following.

\begin{theorem}\label{s:Bstr:Binteg} $B$ is integrable, as a direct sum of the $B_{n}$, which are integrable, and $B$ is a direct sum of quotients of Weyl modules and tensor products of these. \qed
\end{theorem}

\noindent As noted previously, one should consider this as an instance of the quantum Pl\"{u}cker embedding.

\subsection{$B$ is a Nichols algebra}\label{s:Bstr:ss:BNichols}

We complete our analysis of the structure of $B_{1}$, showing that the braided antipode and braided counit on $B_{1}$ are of enveloping algebra type.  We then prove that the braided Hopf algebra $B=B(\mathfrak{T},\mathfrak{J},\iota,q)$ is a Nichols algebra.

We recall from Section~\ref{ss:grandsplproj} that whenever $B$ arises from a split projection associated to an $\nat$-grading, the first homogeneous component $B_{1}$ is a subspace of the braided-primitive elements $\brprim{B}$.  This is a general fact but we point out that in our special case, this is easy to see directly.  Firstly, calculation of the braided coproduct shows that $\bcoproduct(F_{d}H_{d}) \in \brprim{B}$.  Then since $\bcoproduct$ is by definition a morphism in the (braided) module category, it is $\Ad$-invariant.  But $B_{1}$ is generated as a module by the elements $F_{d}H_{d}$ for $d\in D$ so by $\Ad$-invariance, all elements of $B_{1}$ must be braided-primitive.  

The same idea applies to the braided antipode and braided counit, as follows.  For $d \in D$, $\bantipode(F_{d}H_{d})=-F_{d}H_{d}$ and therefore $\bantipode(\Adj{x}{F_{d}H_{d}})=-\Adj{x}{F_{d}H_{d}}$ for any $x\in \qea{J}$.  Hence for all $b\in B_{1}$, we have $\bantipode b=-b$.  Similarly, we have $\bcounit(1)=1$ and $\bcounit(b)=0$ for all $b\in B_{n}$, $n \geq 1$.

For $\mathfrak{J} \subseteq_{\iota} \mathfrak{T}$ a sub-root datum, we have two measures of the ``difference'' between $\mathfrak{T}$ and $\mathfrak{J}$.  

\begin{definition}\label{s:defnofB:d:corank} Let $\mathfrak{J} \subseteq_{\iota} \mathfrak{T}$.
\begin{enumerate}
\item the quantity $| I \setminus \iota(J) |$ will be called the corank of $\mathfrak{J}$ in $\mathfrak{T}$;
\item the quantity $\dim B_{1}$ will be called the index of $\mathfrak{J}$ in $\mathfrak{T}$, denoted $| \mathfrak{T} \colon \mathfrak{J} |_{q}$.
\end{enumerate} \vspace{0em}
\end{definition}

\noindent The definition of corank mimics that in \cite[\S 2]{LieInduction}: it counts the number of nodes deleted from the associated Dynkin diagram.  One might naturally concentrate on the corank one case but the results here do not assume this.

Note that the index depends on $q$ as well as the data $\mathfrak{J} \subseteq_{\iota} \mathfrak{T}$.  More properly, the index is a measure of the ``difference'' between $\qeb{T}$ and $\qeb{J}$ with the latter a subalgebra of the former via a map induced by $\iota$.  Since $q$ is understood to be fixed throughout, we will simply say ``index'' rather than the more cumbersome ``$q$-index''.  Then we may talk of ``finite index'' and note this situation occurs in a large class of examples.

\begin{lemma} Let $\mathfrak{T}$ be a finite type root datum.  Then for every sub-root datum $\mathfrak{J} \subseteq_{\iota} \mathfrak{T}$ and every choice of $q$, the index $| \mathfrak{T} \colon \mathfrak{J} |_{q}$ is finite.
\end{lemma}

\begin{proof}
Each root datum contains the information of a Cartan datum and correspondingly a sub-root datum contains the information of a sub-Cartan datum, defined in the obvious way.  Then since $\mathfrak{T}$ is a finite type Cartan datum, every sub-Cartan datum $\mathfrak{J}$ of $\mathfrak{T}$ is also of finite type.  Hence by \cite[Proposition~6.3.4]{LusztigBook} for every $J$-dominant weight $\lambda'$ the associated Weyl $\qea{J}$-module $L(\lambda')$ is finite-dimensional.  (Recall here the definition of $J$-dominant from Section~\ref{s:Bstr:ss:Bmodstr}.)   Therefore, by Proposition~\ref{s:Bstr:B1integ}, $B_{1}$ is finite-dimensional, since it is a direct sum of finitely many quotients of Weyl modules.
\end{proof}

\noindent This need not be the case in general, for example if $\mathfrak{T}$ is affine type and $\mathfrak{J}$ is finite type.

We also note the following fact:

\begin{lemma}
Let $\mathfrak{J} \subseteq_{\iota} \mathfrak{T}$ and let $B=B(\mathfrak{T},\mathfrak{J},\iota,q)$ be the associated graded Hopf algebra in the braided category of $D(\qebO{T})$-modules.  Then $B$ has finite-dimensional homogeneous components if and only if the index $|\mathfrak{T}\colon \mathfrak{J}|_{q}$ of $\mathfrak{J}$ in $\mathfrak{T}$ is finite.
\end{lemma}

Finally, we can combine all our previous work to prove the following theorem.

\begin{theorem}\label{s:Bstr:BisNichols}
Let $\mathfrak{J} \subseteq_{\iota} \mathfrak{T}$ and let $B=B(\mathfrak{T},\mathfrak{J},\iota,q)$ be the associated graded Hopf algebra in the braided category of $D(\qebO{T})$-modules.  If the index $|\mathfrak{T} \colon \mathfrak{J}|_{q}$ of $\mathfrak{J}$ in $\mathfrak{T}$ is finite then $B$ is a Nichols algebra.
\end{theorem}

\begin{proof}
Recall that we must show the following:
\renewcommand{\labelenumi}{\textup{(\hspace{-1pt}}\textit{\alph{enumi}}\textup{\hspace{-0.5pt})}\ }
\begin{enumerate}
\item $B_{0}=k$,
\item $B_{1}=\brprim{B}$, and
\item $B$ is generated as an algebra by $B_{1}$.
\end{enumerate}
\renewcommand{\labelenumi}{\textit{\roman{enumi}}\textup{)}\ }

Condition (\hspace{-1pt}\textit{a}\hspace{-0.5pt}) was discussed at the end of Subsection~\ref{ss:grandsplproj} and condition (\hspace{-1pt}\textit{c}\hspace{-0.5pt}) is Corollary~\ref{s:Bstr:B1generatesB}.  In order to show (\hspace{-1pt}\textit{b}\hspace{-0.5pt}), we make use of Lemma~\ref{s:prelims:l:Nicholsifdualdeg1gen}, which tells us that if we can show that the graded dual of $B$ is generated in degree one then $B$ satisfies (\hspace{-1pt}\textit{b}\hspace{-0.5pt}).  Note that by the preceding lemma, our hypothesis on the finiteness of the index of $\mathfrak{J}$ in $\mathfrak{T}$ means that the condition of Lemma~\ref{s:prelims:l:Nicholsifdualdeg1gen} regarding the finite-dimensionality of the homogeneous components of $B$ is satisfied.

Now we need to identify the graded dual of $B$.  But we have already found this: it is the braided Hopf algebra $B'$ that occurs in Theorem~\ref{s:defnofB:isotoDB}, such that $\qea{T} \iso B \rtimesdot \widetilde{\qea{J}} \ltimesdot \op{(B')}$.  Both $B$ and $B'$ are $\widetilde{\qea{J}}$-modules and indeed are Hopf algebras in the braided category of such modules.  By the same arguments as for $B$, $B'$ is graded.  Lusztig's pairing is graded and so we see that $B'$ is the graded dual of $B$.  Of course, this is just seeing the symmetry in the positive and negative parts of the quantized enveloping algebras.

Moreover, a proof exactly analogous to that of Theorem~\ref{s:Bstr:alggensofB} shows that $B'$ is generated in degree one: we have a corresponding bosonisation involving $B'$ and the argument there is essentially only dependent on the semi-direct algebra structure of the bosonisation.

Then by Lemma~\ref{s:prelims:l:Nicholsifdualdeg1gen}, we are done and $B$ is a Nichols algebra when the sub-root datum has finite index.
\end{proof}

We remark that the graded dual $B'$ is also a Nichols algebra, by general principles.  So, although we have worked throughout with $\qeb{T}$, we obtain the same results if we consider $\qebplus{T}$.

We conclude with an example.

\begin{example*}  Let $k=\complex$ and take $q\in \dual{\complex}$ such that $q^2\neq 1$.  Let $A_{3}$ be the standard root datum associated to $\Lie{sl}_{4}(\complex)$ and $A_{2}$ that associated to $\Lie{sl}_{3}(\complex)$.  Let $B=B(A_{3},A_{2},\iota,q)$ with $\iota(j)=j$.  Then $B_{1}$ is generated as a $\widetilde{\qea{sl_{\mathrm{3}}}}$-module, with the adjoint action, by $\{ F_{3}K_{3} \}$.  Below, the notation $\qbracket{\ }{\ }{q}$ denotes the $q$-commutator, $\qbracket{x}{y}{q}=xy-qyx$.  A basis for the module is given by $\{ b_{1}, b_{2}, b_{3} \}$ with
\begin{align*}
b_{1} & =F_{3}K_{3}, \\
b_{2} & =\Adj{F_{2}}{b_{1}}=\qbracket{F_{2}}{F_{3}}{q}K_{2}K_{3},\ \text{and} \\
b_{3} & =\Adj{F_{2}}{b_{2}}=\qbracket{F_{1}}{\qbracket{F_{2}}{F_{3}}{q}}{q}K_{1}K_{2}K_{3}.
\end{align*}
The rest of the adjoint action $\Ad$ is as follows:
\begin{align*}
\Adj{E_{2}}{b_{2}} & =b_{1} & \Adj{K_{1}}{b_{1}} & = b_{1} & \Adj{K_{2}}{b_{1}} & = qb_{1} \\
\Adj{E_{1}}{b_{3}} & =b_{2} & \Adj{K_{1}}{b_{2}} & = qb_{2} & \Adj{K_{2}}{b_{2}} & =\inv{q}b_{2} \\
 & & \Adj{K_{1}}{b_{3}} & = \inv{q}b_{3} & \Adj{K_{2}}{b_{3}} & = b_{3} 
\end{align*}
\[ \Adj{F_{1}}{b_{1}} =\Adj{F_{2}}{b_{2}}=\Adj{F_{1}}{b_{3}}=\Adj{F_{2}}{b_{3}}=0 \]
\[ \Adj{E_{1}}{b_{1}} =\Adj{E_{2}}{b_{1}}=\Adj{E_{1}}{b_{2}}=\Adj{E_{2}}{b_{3}}=0 \]

This may be represented graphically as follows, with arrows for non-zero actions (not including the $K_{i}$):
\[ \xymatrix@1@M=5pt{{b_{1}} \ar@<1ex>[r]^{F_{2}} & {b_{2}} \ar@<1ex>[l]^{E_{2}} \ar@<1ex>[r]^{F_{1}} & {b_{3}} \ar@<1ex>[l]^{E_{1}} } \]

Then $B$ is generated as an algebra by $b_{1}$, $b_{2}$ and $b_{3}$, as above, subject to the relations 
\[ \qbracket{b_{i}}{b_{j}}{q} = 0 \quad \mathrm{for}\ i<j. \]
These relations were obtained by explicit calculation, with the assistance of the computer program GAP (\cite{GAP4}).

We also give the braiding $\Psi|_{B_{1}}\colon B_{1} \btensor B_{1} \to B_{1} \btensor B_{1}$ on the basis elements:
\begin{align*}
\Psi(b_{i}\tensor b_{i}) & = q^{-2}b_{i}\tensor b_{i} \quad \text{for $1 \leq i \leq 3$} \\
\Psi(b_{i}\tensor b_{j}) & = q^{-1}b_{j}\tensor b_{i}  \quad \text{for $1 \leq i < j \leq 3$} \\
\Psi(b_{j}\tensor b_{i}) & = (q^{-2}-1)b_{j}\tensor b_{i}+q^{-1}b_{i}\tensor b_{j}  \quad \text{for $1 \leq i < j \leq 3$} 
\end{align*}

\noindent In fact, $\Psi$ is of Hecke type and so $B$ is isomorphic to the quantum symmetric algebra $S_{q}(V)$ for $V$ of dimension 3 over $\complex$, or the quotient of this by $N$th powers of the generators if $q$ is an $N$th root of unity.  This braided Hopf algebra is quadratic and Koszul by \cite[Proposition 3.4]{AndrSchnPHA}.

We see from the above formul\ae\ that in the limit as $q\to 1$, the braiding $\Psi$ becomes $\tau$, the tensor product flip map, $\tau\colon a\tensor b \mapsto b\tensor a$.  Then $\beta(A_{3},A_{2},\iota)=S(V)$, the ordinary symmetric algebra on $V$.  The Lie algebra $\Lie{n}^{-}_{D}$ is the three-dimensional natural $\Lie{sl}_{3}$-module $V$ with the zero Lie bracket (\ie it is an Abelian Lie algebra).  The universal enveloping algebra $U(\Lie{n}^{-}_{D})$ is therefore $S(V)$ and so the interpretation of $B$ as $\qea{n_{\mathit{D}}^{-}}$ is very explicit in this case.
\end{example*}

\section*{Concluding remarks}

In addition to the motivations described in the introduction, there are several other potential applications of these results concerning bases of various types and non-standard quantum groups.  We briefly describe these now.

We concur with Majid (\cite{MajidNewConstr},\cite{QGP}) that an interesting area for future work would be to attempt to identify quantum groups obtained by the inductive construction that do not have classical counterparts---purely quantum phenomena.  One would do this by asking about the possible braided Hopf algebra structures that could be fed into a double-bosonisation to yield an object ``close to'' a quantized enveloping algebra.  (We compare this with the investigations in \cite{LieInduction}, which asked which braided-Lie bialgebras could be used to obtain simple Lie algebras.)  It is certainly not clear how one might make sense of the phrase ``close to'' here, in order to adequately describe the category of objects under consideration.  We would also want to know which properties were shared with quantized enveloping algebras.  However, one might reasonably start by constructing explicit examples, as was done in \cite{MajidNewConstr}, extending $U_{q}(\Lie{su}_{2})$ by a ``fermionic'' braided super-plane $\complex_{q}^{0|2}$ instead of the standard braided plane $\complex_{q}^{2}$ that gives $U_{q}(\Lie{su}_{3})$.  A starting point might be to only consider Nichols algebras of integrable modules of quantized enveloping algebras, \ie the situation we have described here.

We also remark that much of what we have done here is exactly mirrored for the reduced form of $\qea{T}$, $u_{q}(\Lie{T})$, when $q$ is a root of unity (see \cite[Chapter~36]{LusztigBook}).  Then when $u_{q}(\Lie{T})$ is finite-dimensional (\eg for $\Lie{T}$ of finite type), the associated Nichols algebras are also finite-dimensional and hence are subsumed into the large amount of theory developed for this case (see for example \cite{AndrSchnPHA} and \cite{AndrNichols}).  In particular, this work is closely related to that in \cite{AndrSchnFD-PHA}.

An significant outstanding problem is to understand fully the relations in the Nichols algebras we have identified.   We may be aided in this by recent work of Ufer (\cite{UferUqgmod}) on Nichols algebras of $\qea{\Lie{g}}$-modules.  He studies Nichols algebras of integrable modules and in particular their Gel\cprime fand--Kirillov dimension and their defining relations, subject to certain assumptions.  The principal assumption made is that the braiding is of a special type, called strong exponential type.  Also, in earlier work Ufer (\cite{UferPBW}) noted that braidings of the type we obtain here are triangular and that consequently one has Poincar\'{e}--Birkoff--Witt-type (PBW-type) bases on the Nichols algebras, at least in the generic case and for finite types.

Ufer's work and our own are in parallel: the approach here was to try to find the most general results possible about the braided Hopf algebras arising from deleting nodes, whereas Ufer considers Nichols algebras over general integrable modules.  Clearly there is much of common interest but not all of those Nichols algebras considered by Ufer will appear in our setting and not all of the braided Hopf algebras appearing here are known to be Nichols algebras.  We would like to have more precise information about the braiding associated to $B$, especially outside the finite-type case.

In describing the inductive construction of quantum groups, we have had in mind the possibility of using this to provide an alternative approach to the proof of various properties of quantized enveloping algebras.  For example, the existence of Poincar\'{e}--Birkoff--Witt-type (PBW-type) bases was discussed by Ufer.  One can see, as noted by Majid (\cite{MajidInductive}), that once one understands the structure of $B$ then one can re-prove the existence of the PBW-type bases inductively, deleting (or adding) one node at a time.  The study of the canonical basis should also be approachable inductively, with the base case being rank one root data.  Then for the inductive step, one should show that such a basis is induced on $B$.  The interpretation of $B$ as a braided enveloping algebra over a quantized enveloping algebra module (indeed, often a finite-dimensional module) makes this plausible.

\subsection*{Acknowledgements}

I am very grateful to Shahn Majid for the initial suggestions that led to the work here and for much help and encouragement during its completion.  I also gratefully acknowledge the assistance provided by Jonathan Dixon, particularly relating to the representation-theoretic aspects of this work.  I would also like to thank Stefan Kolb for several helpful conversations.

The majority of the work in this paper has appeared in the author's PhD thesis (\cite{Thesis}), completed at Queen Mary, University of London under the supervision of Prof.\ Majid and funded by an EPSRC Doctoral Training Account.  The remainder has been completed during the author's research fellowship at Keble College, Oxford.  I would also like to thank the Mathematical Institute at Oxford for its provision of facilities.

I would also like to thank the referee for correcting an error in a previous version and for several other helpful comments.

\bibliographystyle{elsarticle-num}
\bibliography{references}\label{references}

\end{document}